\documentclass[12pt]{article}
\usepackage{amsmath, amssymb, amsthm, enumerate, booktabs, accents, framed, graphicx, array, color, multirow, mathrsfs, amsrefs, mathrsfs, cases, euscript, bm, multicol, caption, subfigure}

\usepackage[utf8]{inputenc}
\usepackage[colorlinks,citecolor=blue,urlcolor=blue]{hyperref}
\usepackage[usenames,dvipsnames,svgnames,table]{xcolor}
\usepackage{a4wide}

\allowdisplaybreaks[4]

\numberwithin{equation}{section} \theoremstyle{plain}
\newtheorem{theorem}{Theorem}[section]
\newtheorem{lemma}{Lemma}[section]
\newtheorem{corollary}{Corollary}[section]
\newtheorem{proposition}{Proposition}[section]

\newtheorem{definition}{Definition}
\newtheorem{remark}{Remark}[section]

\def\bC{\mathbb C}
\def\bE{\mathbb E}
\def\bN{\mathbb N}
\def\bR{\mathbb R}

\def\cC{\mathcal C}
\def\cI{\mathcal I}
\def\cJ{\mathcal J}
\def\cP{\mathcal P}

\def\y{\mathbf y}

\def\Tr{\mathrm {Tr}}

\def\Var{\mathrm {Var}}
\def\deg{\mathrm {deg}}

\begin{document}

\title{On spectrum of sample covariance matrices from large tensor vectors}

\date{\today}
\author{Wangjun Yuan\thanks{Department of Mathematics, University of Luxembourg. E-mail: ywangjun@connect.hku.hk}}
\maketitle

\begin{abstract}
    In this paper, we investigate the limiting empirical spectral distribution (LSD) of sums of independent rank-one $k$-fold tensor products of $n$-dimensional vectors as $k,n \to \infty$. Assuming that the base vectors are complex random variables with unit modular, we show that the LSD is the Mar\v{c}enko-Pastur law. Comparing with the existing results, our limiting setting allows $k$ to grow much faster than $n$. Consequently, we obtain the necessary and sufficient conditions for Mar\v{c}enko-Pastur law to serve as the LSD of our matrix model. Our approach is based on the moment method. 
\end{abstract}


\medskip 


\section{Introduction}
\label{sec:intro}

For $n \in \bN$, let $\y = \frac{1}{\sqrt{n}} (\xi_1, \ldots, \xi_n) \in \bC^n$, where $\{\xi_1, \ldots, \xi_n\}$ is a family of i.i.d. centered random variables with unit variance, and let $\{\y_{\alpha}^{(l)}: 1 \le \alpha \le m, 1 \le l \le k\}$ be a family of i.i.d. copies of $\y$.
For $1 \le \alpha \le m$, define $Y_{\alpha} = \y_{\alpha}^{(1)} \otimes \cdots \otimes \y_{\alpha}^{(k)}$ yielding a $k$-fold tensor product. We identify each $Y_{\alpha}$ as an $n^k$-dimensional vector, and denote by $Y = (Y_1, \ldots, Y_m)$. Let $\{\tau_1, \tau_2, \ldots\}$ be a sequence of real numbers, we consider the sum of $m$ independent rank-$1$ Hermitian matrices:
\begin{align}\label{eq:matrix}
	M_{n,k,m} = \sum_{\alpha=1}^m \tau_{\alpha} Y_{\alpha} Y_{\alpha}^*,
\end{align}
which is an $n^k \times n^k$ Hermitian matrix.

In statistics, the model \eqref{eq:matrix} is called the sample covariance matrix. It helps to understand the population covariance matrix. The limiting empirical spectral distribution (LSD) of $M_{n,k,m}$ can serve as a test statistics. The simplest case $k=1$, which corresponds to the population vector with i.i.d. entries, was well-studied in the literature. Under appropriate moment conditions on $\xi_1$, the LSD was obtained in the seminal paper \cite{MP67} when $n \to +\infty$ and $m/n^k \to c$ for some positive constant $c$. When $\tau_\alpha=1$ for all $\alpha$, the resulting LSD is the famous Mar\v{c}enko-Pastur law. This probability distribution has a density function given by
\begin{align} \label{eq-MP law}
    p(x) = \dfrac{\sqrt{((1+\sqrt{c})^2-x)(x-(1-\sqrt{c})^2)}}{2\pi x} {\bf 1}_{[(1-\sqrt{c})^2,(1+\sqrt{c})^2]}(x) + (1-c) \delta_0(dx) {\bf 1}_{0<c<1}.
\end{align}
Further details can be found in \cites{Anderson2010, Bai2010}.
Subsequently, numerous works on the Mar\v{c}enko-Pastur law have emerged in the literature, including \cite{BaiZhou08}, \cite{Pajor09}, \cite{Silv95}, and the reference therein. Later, the necessary and sufficient conditions on $Y_1$ were carried out in \cite{Yaskov2016} to ensure that the Mar\v{c}enko-Pastur law is the LSD of $M_{n,1,m}$ when $\tau_\alpha=1$ for all $\alpha$.

Another motivation to study the model \eqref{eq:matrix} comes from quantum information theory.
We consider a classical probability problem of distributing $m$ balls randomly into $n$ bins. Using the terminology of matrices, this classical probability problem is equivalent to the spectrum of the $M_{n,1,m}$ when $Y_\alpha$ is chosen randomly from the canonical basis $\{e_1, \ldots, e_n\}$ of $\bC^n$. The eigenvalues of the $M_{n,1,m}$ are the frequency of occurrence of vectors $e_1, \ldots, e_n$ in the family $\{Y_\alpha: 1 \le \alpha \le m\}$.
It was considered as a quantum analog for $M_{n,1,m}$ if the vector $Y_\alpha$ is picked randomly from the unit sphere on $\bC^n$.
A modified version of the quantum problem, which was introduced in \cite{Hastings2012}, is to choose the vector $Y_1$ from the random product states in $(\bC^n)^{\otimes k}$. When $k$ and $m/n^k$ are fixed, and the base vector $\y$ is Gaussian, \cite{Hastings2012} established the convergence in expectation of the ESD towards the Mar\v{c}enko-Pastur law \eqref{eq-MP law} as $n \to +\infty$ by computing the moments.
In the realm of quantum physics and quantum information theory, it is natural to explore systems with a multitude of quantum states.
The paper \cite{Collins2011} characterizes the quantum entanglement of structured random states and studies the spectral density of the reduced state when the number of the quantum states $k$ is large.
Besides, the asymptotic behavior of the average entropy of entanglement for elements of an ensemble of random states associated with a graph was studied in \cite{Collins2013} when the dimension of the quantum subsystem is large.

The cases $k \ge 2$ and $k=1$ are very different. The tensor structure appears when $k \ge 2$, which results in the dependence of the entries of $Y_\alpha$ and non-unitary invariant. In a recent work \cite{Lytova2018}, the LSD for the $k$-fold tensor model $M_{n,k,m}$ was explored when $k$ is large. In the special case where $\tau_{\alpha} \equiv 1$, the LSD precisely follows the Mar\v{c}enko-Pastur law \eqref{eq-MP law}. A central limit theorem (CLT) is also established for a class of linear spectral statistics following the approach of \cite{LytovaPastur09} when $k=2$. The main setup in \cite{Lytova2018} is that the number of tensor folds ($k$) must be sufficiently small in comparison to the spatial dimension ($n$). More precisely, $k/n \to 0$ is required for the validity of the LSD. In a more recent study, \cite{BYY2022} considered the case where $k = O(n)$ for the model \eqref{eq:matrix}. Under the condition $k/n \to d$, they derived the limiting moment sequence of the empirical spectral distribution (ESD) of $M_{n,k,m}$, assuming the finiteness of the all moments of $\xi_1$. It is interesting that, the fourth moment of $\xi_1$ contributes to the limiting moment sequence. If $\tau_{\alpha} \equiv 1$, the resulting limiting moment sequence corresponds to the Mar\v{c}enko-Pastur law \eqref{eq-MP law} when either $d=0$ or $d>0$ and the fourth moment of $\xi_1$ is 1.

In the present paper, we study the model \eqref{eq:matrix} with $\xi_1$ chosen from the unit circle on the complex plane. Our focus lies in the scenario where $k$ goes to infinity much faster than the setting in \cite{BYY2022}. We derive the limiting moment sequence of the ESD of $M_{n,k,m}$ when $n,k$ goes to infinity. When $\tau_{\alpha} \equiv 1$, the limiting moment sequence reveals that the LSD is exactly Mar\v{c}enko-Pastur law \eqref{eq-MP law}. From the point of view of probability theory, we remove the restriction on the speed of $k$ approaching infinity, as required in \cites{Lytova2018,BYY2022}. Our results even allows that $k/n$ does not have a limit when $n \to \infty$. In practical terms, one would anticipate that the dimension of a system remains fixed and is reused multiple times, which leads to the scenario that $k$ tends to infinity while $n$ is fixed. The study of the model \eqref{eq:matrix} with fixed $n$ remains an open question, and we plan to address this case in our future work. As an variant of this real-world scenario, the results in this paper apply to the setting where $k \gg n \gg 1$.

We would like to remark that the tensor $Y_\alpha$ introduced above is the non-symmetric random tensor model. Instead of considering the tensor product of i.i.d. vectors, the $k$-fold tensor product $\y^{\otimes k}$ of the same random vector $\y \in \bC^n$ is known as the symmetric random tensor model. The limiting spectral distribution of the Hermitian matrix \eqref{eq:matrix} constructed by i.i.d. copies of $\y^{\otimes k}$ was studied in \cites{Vershynin2021, Yaskov2023}.

Throughout the paper, we assume that the parameters $m,n,k$ grow towards infinity following the proportion
\begin{align} \label{eq-def-ratio}
	\dfrac{m}{n^k} \to c
\end{align}
for some constant $c \in (0,\infty)$. The following theorem is the first main results of the paper, where we establish the convergence in expectation of the moments.

\begin{theorem} \label{Thm-main}
Let $M_{n,k,m}$ be in \eqref{eq:matrix} with $|\xi_1|=1$. Suppose that for all $q \in \bN$,
\begin{align} \label{eq-condition-moment convergence}
    \dfrac{1}{m} \sum_{j=1}^m \tau_j^q \to m_q^{(\tau)}, \quad m \to +\infty.
\end{align}
Assume that \eqref{eq-def-ratio} holds. Then for any fixed $p \in \bN_+$, we have
\begin{align*}
    \lim_{n,k \to +\infty} \dfrac{1}{n^k} \bE \big[ \Tr M_{n,k,m}^p \big]
   = \sum_{s=1}^p c^s \sum_{\alpha \in \cC_{s,p}^{(1)}} \left( \prod_{t=1}^s m_{\deg_t(\alpha)}^{(\tau)} \right).
\end{align*}
Here, $\deg_t(\alpha)$ is the frequency of $t$ in the sequence $\alpha$ and is given by \eqref{eq:deg_t}, and $\cC_{s,p}^{(1)}$ is a set of sequences that is defined in Lemma \ref{lem-Bai}.
\end{theorem}

Our approach is based on the method of moments. We associate graphs to each terms of the moment, and compute the moment by distinguishing the graphs that contributes to the limit.
The class $\cC_{s,p}^{(1)}$ of sequences corresponds to the largest leading terms, while the sequences in $\alpha \notin \cC_{s,p}^{(1)}$ are negligible.
It's important to note that the tensor structure introduces a power of $k$ in the moment calculation.
For the case $k=1$, the limit of the moment sequence can be obtained by counting the size of $\cC_{s,p}^{(1)}$, since only the leading terms in the sum of $\alpha \in \cC_{s,p}^{(1)}$ contribute to the limit. We refer the interested readers to \cite{Bai2010}*{Section 3.1.3}. This is also true whenever $k=o(n)$. See \cite{BYY2022}*{Section 4}.
For the case $k=O(n)$ studied in \cite{BYY2022}, besides the first leading term, the second leading term also contributes to the limit.
In our current setting, where $k$ can grow much faster, all terms associated with $\alpha \in \cC_{s,p}^{(1)}$ may contribute to the limit. Hence, we need to characterize all the possible graphs associate with $\alpha \in \cC_{s,p}^{(1)}$. This is the main novelty in methodology of the present paper. 

In the next theorem, we strengthen Theorem \ref{Thm-main} from convergence in expectation to almost sure convergence.

\begin{theorem} \label{Thm-main2}
Assume that the conditions in Theorem \ref{Thm-main} hold. Suppose that $k = k(n)$ is a function of $n$ and tends to infinity as $n \to \infty$. Then for any fixed $p \in \bN_+$,
\begin{align*}
    \lim_{n \to +\infty} \dfrac{1}{n^k} \Tr M_{n,k,m}^p
   = \sum_{s=1}^p c^s \sum_{\alpha \in \cC_{s,p}^{(1)}} \left( \prod_{t=1}^s m_{\deg_t(\alpha)}^{(\tau)} \right),
\end{align*}
almost surely.
\end{theorem}

After establishing the limiting moment sequence, it is nature to inquire whether this sequence uniquely characterizes a probability measure. The following corollary provides a condition that guarantees the uniqueness of the probability measure corresponding to the moment sequence, which leads to the almost sure convergence of the ESD of $M_{n,k,m}$. 

\begin{corollary} \label{Coro}
Assume that the conditions in Theorem \ref{Thm-main} hold. Suppose that there exists a positive constant $A$, such that $|m_q^{(\tau)}| \le A^qq^q$ for all $q \in \bN$. Then there exists a probability measure $\mu$ whose moment sequence is
\begin{align} \label{eq-moment}
    \int_{\bR} x^p \mu(dx) = \sum_{s=1}^p c^s \sum_{\alpha \in \cC_{s,p}^{(1)}} \left( \prod_{t=1}^s m_{\deg_t(\alpha)}^{(\tau)} \right), \quad \forall p \in \bN_+.
\end{align}
Moreover, suppose that $k = k(n)$ is a function of $n$ and tends to infinity when $n \to \infty$, then the ESD of $M_{n,k,m}$ converges almost surely to $\mu$.

In particular, if $\tau_\alpha = 1$ for all $1 \le \alpha \le m$, then the ESD of $M_{n,k,m}$ converges almost surely to the Mar\v{c}enko-Pastur law \eqref{eq-MP law}.
\end{corollary}

By combining the results in \cite{Lytova2018}*{Theorem 1.2}, \cite{BYY2022}*{Theorem 2.1} and Corollary \ref{Coro}, we immediately obtain the following necessary and sufficient condition for Mar\v{c}enko-Pastur law \eqref{eq-MP law} to be the LSD of the model \eqref{eq:matrix} when $\tau_\alpha \equiv 1$.  

\begin{corollary} \label{Coro'}
Let $M_{n,k,m}$ be in \eqref{eq:matrix} with $\tau_\alpha = 1$ for all $1 \le \alpha \le m$. Then the ESD of $M_{n,k,m}$ converges almost surely to the Mar\v{c}enko-Pastur law \eqref{eq-MP law} if and only if either $k=o(n)$ or $|\xi_1|=1$.
\end{corollary}

The rest of the paper is organized as follow. We develop the theory of graph combinatories in Section \ref{sec:graph}. We first introduce some results from literature on the graph combinatorics in Subsection \ref{sec-preliminary}. Then we characterize the set $\cC_{s,p}^{(1)}$, which corresponding to the leading term in the moment computation. The paired graph, which is the graph that contributes to the limit for $\alpha \in \cC_{s,p}^{(1)}$, is studied in Section \ref{sec:paired garph}. In Section \ref{sec:moment}, we prove the main theorems. The proofs of Theorem \ref{Thm-main}, Theorem \ref{Thm-main2} and Corollary \ref{Coro} are presented in Section \ref{sec:(a)}, Section \ref{sec:(b)} and Section \ref{sec:(c)}, respectively.

\section{Graph combinatorics} \label{sec:graph}

In this section, we study the graph combinatorics, which will be used in Section \ref{sec:moment}.

\subsection{Preliminaries} \label{sec-preliminary}

In this subsection, we introduce some preliminaries on graph combinatorics, which can be found in \cites{Bai2010, BYY2022}.

For a positive integer $s$, we denote by $[s]$ the set of integers from 1 to $s$. We call $\alpha=(\alpha_1, \ldots, \alpha_p) \in [m]^p$ a sequence of length $p$ with vertices $\alpha_j$ for $1 \le j \le p$. We denote by $|\alpha|$ the number of distinct elements in $\alpha$. If $s = |\alpha|$, then we call $\alpha$ an \emph{$s$-sequence}. Let  $\cJ_{s,p}(m)$ be the set of all $s$-sequences $\alpha \in [m]^p$. Then
\begin{align} \label{eq-0.3}
	[m]^p = \bigcup_{s=1}^p \cJ_{s,p}(m).
\end{align}
For a sequence $\alpha = (\alpha_1, \ldots, \alpha_p)$ and each value $t$ in $\alpha$, we count its frequency by
\begin{align}\label{eq:deg_t}
  \deg_t(\alpha) = \# \{j \in [p]: \alpha_j = t \}.
\end{align}
where we use the notation $\#S$ for the number of elements in the set $S$.

Two sequences are \emph{equivalent} if one becomes the other by a suitable permutation on $[m]$.
The sequence $\alpha$ is \emph{canonical} if $\alpha_1=1$ and $\alpha_u \le \max\{\alpha_1, \ldots, \alpha_{u-1}\} + 1$ for $u \ge 2$. We denote by $\cC_{s,p}$ the set of all canonical $s$-sequences of length $p$.
From the definition above, one can see that the set of distinct vertices of a canonical $s$-sequence is $[s]$. Denote by $\cI_{s,m}$ the set of injective maps from $[s]$ to $[m]$. For a canonical $s$-sequence $\alpha$ and a map $\varphi \in \cI_{s,m}$, we call $\varphi(\alpha)$ the $s$-sequence $(\varphi(\alpha_1), \ldots, \varphi(\alpha_p))$. For each canonical $s$-sequence, its image under the maps in $\cI_{s,m}$ gives all its equivalent sequences, and hence its equivalent class of sequences in $[m]^p$ has exactly $m(m-1) \cdots (m-s+1)$ distinct elements.

We fixed a canonical $s$-sequence $\alpha = (\alpha_1, \ldots, \alpha_p) \in [m]^p$. For $i = (i^{(1)}, \ldots, i^{(p)}) \in [n]^p$, draw two parallel lines referred as the $\alpha$-line and the $i$-line, respectively. Plot $i^{(1)}, \ldots, i^{(p)}$ on the $i$-line and $\alpha_1, \ldots, \alpha_p$ on the $\alpha$-line. Draw $p$ down edges from $\alpha_u$ to $i^{(u)}$ and $p$ up edges from $i^{(u)}$ to $\alpha_{u+1}$ for $1 \le u \le p$ with the convention that $\alpha_1 = \alpha_{p+1}$. We denote the graph by $g(i,\alpha)$ and call such graph a $\Delta(p;\alpha)$-graph. From the definition, one can easily see that the graph $g(i,\alpha)$ is a connected directed graph with up edges and down edges appear alternatively. An example of the $\Delta(p;\alpha)$-graph is given in (a) of the Figure \ref{Fig-1}.

Two graphs $g(i,\alpha)$ and $g(i',\alpha)$ are called \emph{equivalent} if the two sequences $i$ and $i'$ are equivalent, and we write $g(i,\alpha) \sim g(i',\alpha)$ for this equivalence.
For each equivalent class, we choose the canonical graph such that $i = (i^{(1)}, \ldots, i^{(p)}) \in [n]^p$ is a canonical $r$-sequence for some $r \in \bN_+$. A canonical $\Delta(p;\alpha)$-graph is denoted by $\Delta(p,r,s;\alpha)$ if it has $r$ noncoincident $i$-vertices and $s$ noncoincident $\alpha$-vertices.

We call a graph $\Delta_1(p,s;\alpha)$-graph if it is a $\Delta(p;\alpha)$-graphs such that each down edge coincide with exactly one up edge and if we glue the pair of coincident edges and remove the orientation, the resulting graph is a tree with $p$ edges and $p+1$ vertices. Hence, we have $r+s=p+1$. We give an example of $\Delta_1(p,s;\alpha)$-graph in (b) of Figure \ref{Fig-1}.

\begin{figure}[h]
\centering
\subfigure[$\Delta(p,\alpha)$-graph with $p=3$, $\alpha=(1,2,2)$, $i=(1,2,3)$.]{
\begin{minipage}{7cm}
\centering
\includegraphics[scale=0.35]{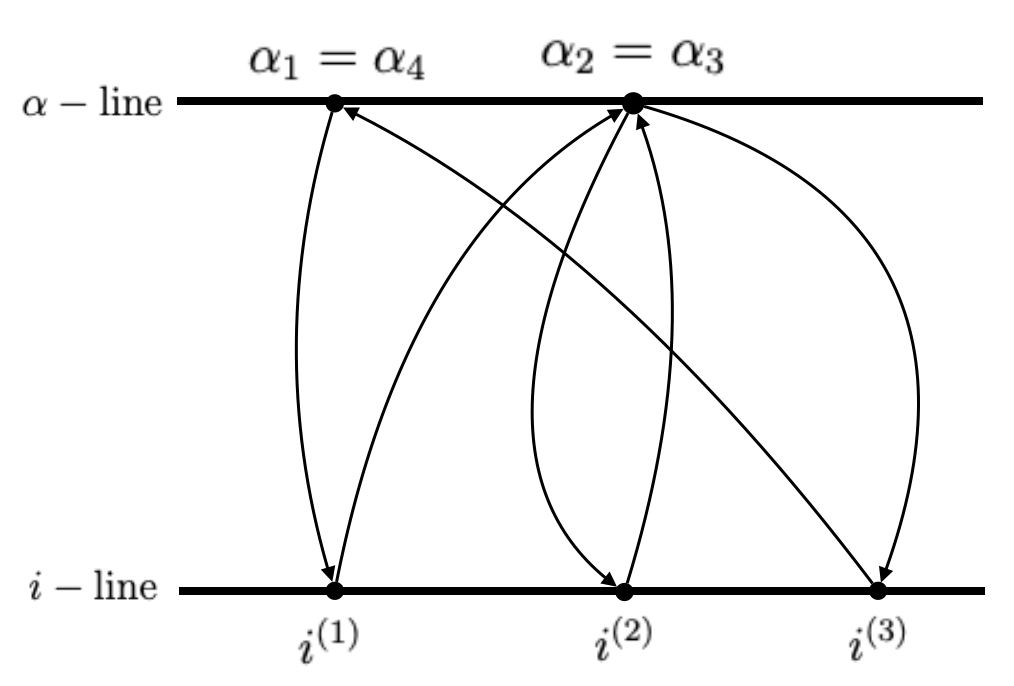}
\end{minipage}
}
\hspace{0.5cm}
\subfigure[$\Delta_1(p,s;\alpha)$-graph with $p=3, s=2$, $\alpha=(1,2,2)$, $i=(1,2,1)$.]{
\begin{minipage}{7cm}
\centering
\includegraphics[scale=0.35]{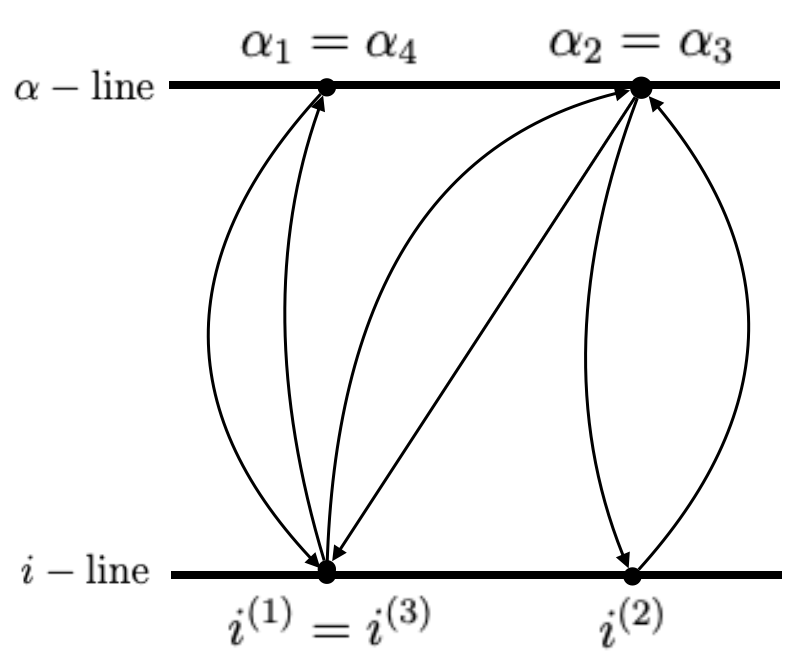}
\end{minipage}
}
\caption{}
\label{Fig-1}
\end{figure}

For a given sequence $\alpha \in \cC_{s,p}$, the following lemma determines the number of sequences $i \in \cC_{p+1-s,p}$ such that $g(i,\alpha) \in \Delta_1(p,s;\alpha)$.

\begin{lemma} {(\cite{BYY2022}*{Lemma 3.1})} \label{lem-Bai}
For any $1 \le s \le p$ and any sequence $\alpha \in \cC_{s,p}$, there is at most one sequence $i \in \cC_{p+1-s,p}$ such that $g(i,\alpha) \in \Delta_1(p,s;\alpha)$. We denote by $\cC_{s,p}^{(1)}$ the set of such canonical sequences $\alpha$. Then the number of the elements in $\cC_{s,p}^{(1)}$ is
\begin{align*}
	\dfrac{1}{p} \binom{p}{s-1} \binom{p}{s}.
\end{align*}
\end{lemma}

\subsection{Characterization of $\cC_{s,p}^{(1)}$}

In this subsection, we establish some properties of the sequences in $\cC_{s,p}^{(1)}$. We start with the following definition.

\begin{definition}
A sequence $\alpha=(\alpha_1, \ldots, \alpha_p)$ is called a \emph{crossing} sequence if there exist $j_1<j_2<j_3<j_4$, such that $\alpha_{j_1} = \alpha_{j_3} \neq \alpha_{j_2} = \alpha_{j_4}$. A sequence is called \emph{non-crossing} sequence if it is not a crossing sequence.
\end{definition}

The following theorem is a characterization of the set $\cC_{s,p}^{(1)}$.

\begin{theorem} \label{Thm-non-crossing}
For any $1 \le s \le p$, the set of all non-crossing sequences $\alpha \in \cC_{s,p}$ is $\cC_{s,p}^{(1)}$.
\end{theorem}

The proof of Theorem \ref{Thm-non-crossing} follows directly from the two lemmas below.

\begin{lemma} {(\cite{BYY2022}*{Lemma 3.4})} \label{Lem-crossing}
For any $1 \le s \le p$ and any sequence $\alpha \in \cC_{s,p}$, if $\alpha$ is a crossing sequence, then $\alpha \notin \cC_{s,p}^{(1)}$.
\end{lemma}

\begin{lemma} \label{Lem-non crossing}
For any $1 \le s \le p$ and any sequence $\alpha \in \cC_{s,p}$, if $\alpha$ is a non-crossing sequence, then $\alpha \in \cC_{s,p}^{(1)}$.
\end{lemma}



\begin{proof} (of Lemma \ref{Lem-non crossing})
We prove by induction on $p$. The case $p=1$ is trivial, since $\alpha = (\alpha_1)$ and $i=(i^{(1)})$, so $g(i,\alpha)$ is the graph with exactly one up edge from $i^{(1)}$ to $\alpha_1$ and one down edge from $\alpha_1$ to $i^{(1)}$.

Assume that Lemma \ref{Lem-non crossing} holds for sequence of length at most $p-1$ for $p \ge 2$. We need to show Lemma \ref{Lem-non crossing} holds for any non-crossing sequence $\alpha=(\alpha_1, \ldots, \alpha_p) \in \cC_{s,p}$. We consider the following two cases according to whether $\alpha_1$ coincide with other vertices.

\noindent{\bf Case 1.} There exists $1<j<p+1$, such that $\alpha_j = \alpha_1$.

In this case, we split the sequence $\alpha$ to two subsequences $\alpha'=(\alpha_1,\ldots,\alpha_{j-1})$ and $\alpha''=(\alpha_j,\ldots,\alpha_p)$. For the subsequence $\alpha'$, it is canonical $s'$-sequence for some $s'<s$. Moreover, $\alpha'$ is non-crossing and has length $j-1 \le p-1$, so by induction hypothesis, we have $\alpha' \in \cC_{s',j-1}^{(1)}$, and thus, there exists a canonical $(j-s')$-sequence $i'$ of length $j-1$, such that $g(i',\alpha') \in \Delta_1(j-1,s';\alpha')$.

For the subsequence $\alpha''$, it is also non-crossing but not canonical. The non-crossing property allows us to identify $\alpha''$ to a canonical sequence. Note that the vertices of $\alpha'$ take values in $[s']$, so the vertices of $\alpha''$ take values in $\{1\} \cup \{s'+1, \ldots,s\}$. We define $\beta''=(\beta_j,\ldots,\beta_p)$ by setting $\beta_k=\alpha_k$ if $\alpha_k=\alpha_1$, and $\beta_k=\alpha_k-s'+1$ if $\alpha_k\not=\alpha_1$. Now $\beta''$ is canonical non-crossing $(s-s'+1)$-sequence of length $p-j+1 \le p-1$. Hence, by induction hypothesis, we have $\beta'' \in \cC_{s-s'+1,p-j+1}^{(1)}$ and there exists a canonical $(p-j-s+s'+1)$-sequence $i''$ of length $p-j+1$, such that $g(i'',\beta'') \in \Delta_1(p-j+1,s-s'+1;\beta'')$.

The sequence $i$ which satisfies $g(i,\alpha) \in \Delta_1(p,s;\alpha)$ can be obtain by 'gluing' the two sequence $i'$ and $i''$. We provide the Figure \ref{Fig-Case 1} part (a) for the idea of gluing two graphs. More precisely, we define a canonical $(p-s+1)$-sequence $i=(i^{(1)},\ldots,i^{(p)})$ by $i^{(k)} = i'^{(k)}$ for $1\le k\le j-1$, and $i^{(k)} = i''^{(k)}+j-s'$ for $j\le k \le p$. One can easily check that $i$ is a canonical $(p-s+1)$ sequence, and the subsequence $(i^{(1)},\ldots,i^{(j-1)})$ and $(i^{(j)},\ldots,i^{(p)})$ has no common vertex. Thus, the subgraph of $g(i,\alpha)$ from $\alpha_1$ to $\alpha_j$ and the subgraph from $\alpha_j$ to $\alpha_{p+1}$ do not have coincide edge and satisfy the definition of $\Delta_1(p,s;\alpha)$-graph. Therefore, we can conclude that the graph $g(i,\alpha) \in \Delta_1(p,s;\alpha)$ so $\alpha \in \cC_{s,p}^{(1)}$.

\begin{figure}[h]
\centering
\subfigure[Combine two $\Delta_1(p,s;\alpha)$-graphs $g(i',\alpha')$ and $g(i'',\alpha'')$.]{
\begin{minipage}{7cm}
\centering
\includegraphics[scale=0.35]{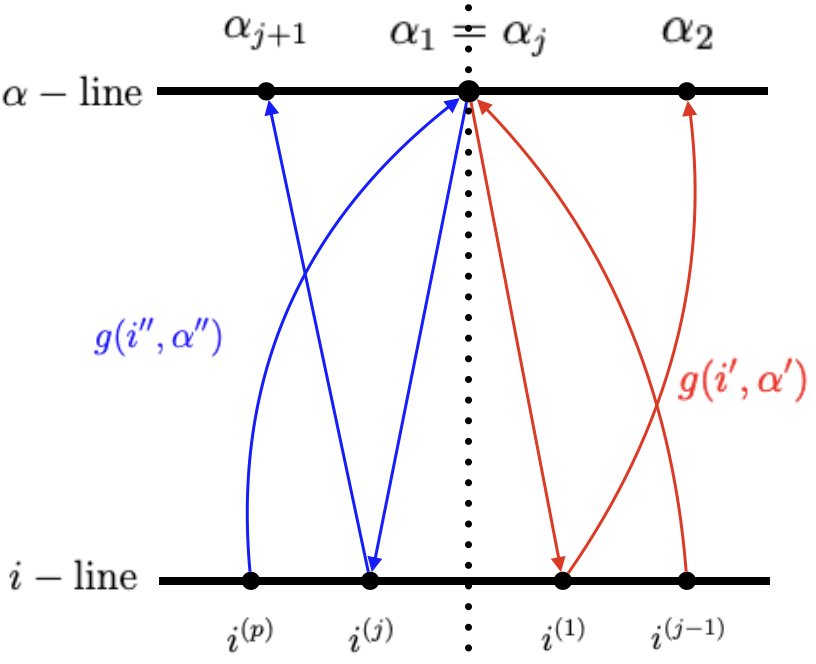}
\end{minipage}
}
\hspace{0.5cm}
\subfigure[Insert the coincident edges $i^{(1)} \to \alpha_2$ and $\alpha_2 \to i^{(2)}$.]{
\begin{minipage}{7cm}
\centering
\includegraphics[scale=0.35]{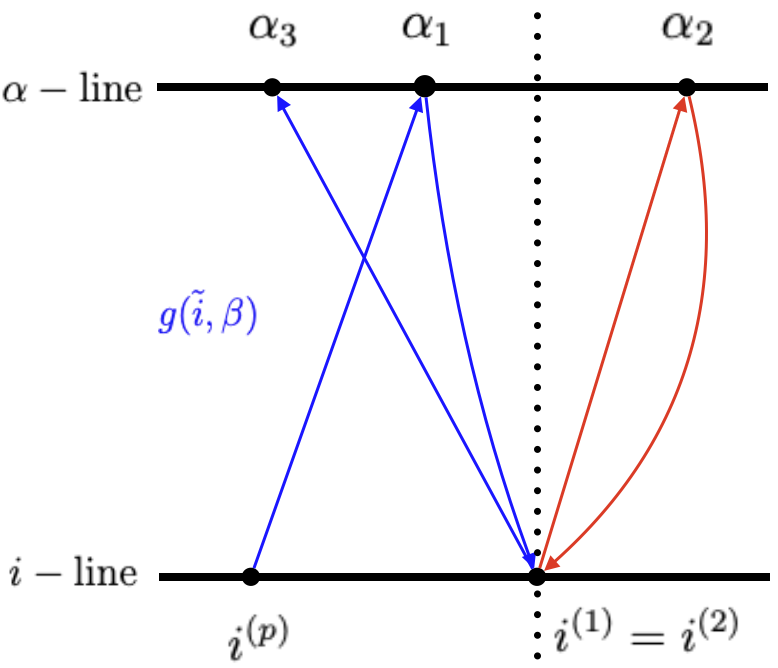}
\end{minipage}
}
\caption{}
\label{Fig-Case 1}
\end{figure}

\noindent{\bf Case 2.} For any $1<j<p+1$, $\alpha_j \not= \alpha_1$. In this case, we have $\alpha_1=1, \alpha_2=2$. We consider the following two subcases.

\noindent{\bf Case 2(a).} If for any $2<k<p+1$, $\alpha_k \not= \alpha_2$. We consider the sequence $\beta = (\beta_1,\ldots,\beta_{p-1})$ given by $\beta_1=\alpha_1$ and $\beta_k = \alpha_{k+1}-1$. Then $\beta$ is a non-crossing canonical $(s-1)$-sequence of length $p-1$. By induction hypothesis, $\beta \in \cC_{s-1,p-1}^{(1)}$, and $g(\tilde i,\beta) \in \Delta_1(p-1,s-1;\beta)$ for some canonical $(p-s+1)$-sequence $\tilde i = (\tilde i^{(1)},\ldots,\tilde i^{(p-1)})$. The sequence $i$ which satisfies $g(i,\alpha) \in \Delta_1(p,s;\alpha)$ can be obtained by 'inserting' the vertex of value 1 to the sequence $\tilde i$ between $\tilde i^{(1)}$ and $\tilde i^{(2)}$. The part (b) of Figure \ref{Fig-Case 1} is provided for the idea of inserting the coincident edges between $i^{(1)} = i^{(2)}$ and $\alpha_2$. More precisely, we define a canonical $(p-s+1)$-sequence $i=(i^{(1)},\ldots,i^{(p)})$ by $i^{(1)} = i^{(2)}=1$ and $i^{(k)} = \tilde i^{(k-1)}$ for $3 \le k\le p$. One can easily check that $i$ is a canonical $(p-s+1)$-sequence of length $p$, and there are exactly two edges with vertex $\alpha_2$: an up edge $i^{(1)} \to \alpha_2$ and a down edge $\alpha_2 \to i^{(2)}$. Noting that $i^{(2)} = i^{(1)}$, the up edge and down edge coincide, but they do not coincide with other edges. Thus, we have $g(i,\alpha) \in \Delta_1(p,s;\alpha)$, which means that $\alpha \in \cC_{s,p}^{(1)}$.

\noindent{\bf Case 2(b).} If there exist $2<k\le p$, such that $\alpha_k=\alpha_2$. Then we can split the sequence $\alpha$ into two subsequences $\alpha',\alpha''$, where $\alpha'=(\alpha_2,\ldots,\alpha_{k-1})$, and $\alpha''=(\alpha_1,\alpha_{k+1},\ldots,\alpha_p)$. One can use the argument of Case 1 to deduce that there are two canonical sequences $i_1,i_2$ of length $k-2$ and $p-k+1$ respectively, such that the graph $g(i_1,\alpha')$ and $g(i_2,\alpha'')$ satisfy the definition of $\Delta_1$-graph. We shift the sequence $i_1$ by adding 1 to the value of each vertex, and denote by $i_1'$ the sequence after shifting. We also shift the sequence $i_2$ by adding $|i_1|$ to all vertices that do not have value $1$. We write $i_2'$ for the sequence after shifting. Then one can glue the two sequences $i_1'$ and $i_2'$ using the argument in Case 1. More precisely, the canonical sequence $i = (i^{(1)},\ldots,i^{(p)})$ can be defined by $i^{(1)} = 1$, $i^{(j)} = i_1'^{(j-1)}$ for $2 \le j \le k-1$, and $i^{(j)} = i_2'^{(j-k+1)}$ for $k \le j \le p$. The non-crossing of the sequence $\alpha$ ensure that the graph $g(i,\alpha) \in \Delta_1(p,s;\alpha)$.
\end{proof}

In the following, we study the graph $g(i,\alpha)$ for non-crossing sequence $\alpha$. We first introduce the conception of paired graph and single graph.

\begin{definition} \label{Def-paired}
Let $\alpha,i$ be two sequences. The $\Delta(p;\alpha)$-graph $g(i,\alpha)$ is called a \emph{paired} graph if for any two vertices, between which the number of up edges equals to the number of down edges. The graph $g(i,\alpha)$ is called a \emph{single} graph if there exist two vertices, such that difference of the number of up edges and down edges between the two vertices is exactly one.
\end{definition}

\begin{remark}
For a $\Delta(p;\alpha)$-graph $g(i,\alpha)$, if one reduces the graph by removing an up edges with one of the coincident down edges at the same time (but keep the vertices), then a paired graph is the graph which can be reduced to a graph without edges, while a single graph is the graph that can be reduced to a graph with at least one single edge.
\end{remark}

\begin{remark}
\begin{enumerate}
    \item A $\Delta_1(p,s;\alpha)$-graph is always a paired graph. Figure \ref{Fig-2} provides two examples of paired graphs that are not $\Delta_1(p,s;\alpha)$-graph.
    
    \item Single graphs exist for any sequence $\alpha$. Figure \ref{Fig-1} (a) is an example of single graph. Indeed, one only need to choose $i$ to have distinct vertices.
    
    \item There are $\Delta(p;\alpha)$-graphs which is neither a paired graph nor a single graph. See for example Figure \ref{Fig-3} where the multiple edges in $g(i,\alpha)$ have the same orientation.
\end{enumerate}
\end{remark}

\begin{figure}[h]
\centering
\subfigure[paired graph with $p=2$, $\alpha=i=(1,1)$.]{
\begin{minipage}{7cm}
\centering
\includegraphics[scale=0.35]{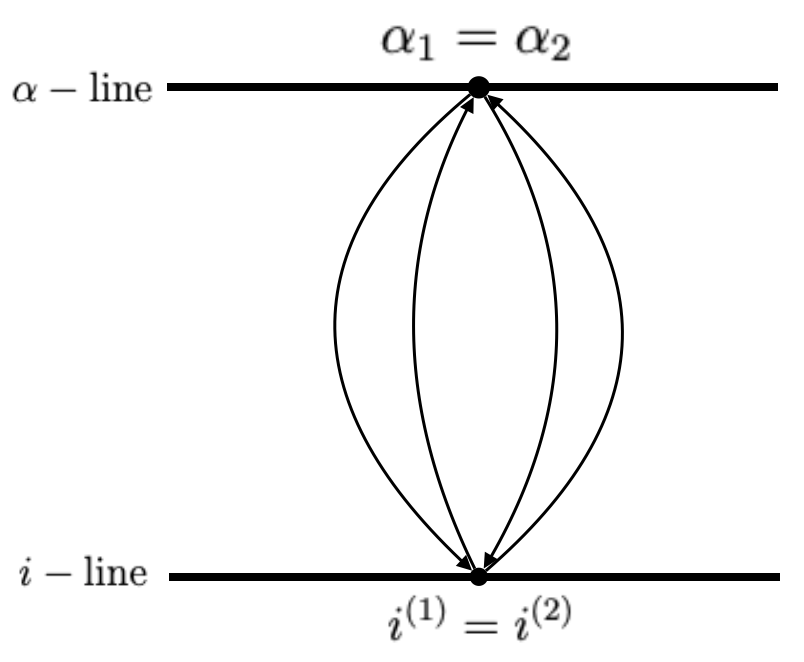}
\end{minipage}
}
\hspace{0.5cm}
\subfigure[paired graph with $p=4$, $\alpha=(1,2,1,2)$, $i=(1,2,2,1)$.]{
\begin{minipage}{7cm}
\centering
\includegraphics[scale=0.35]{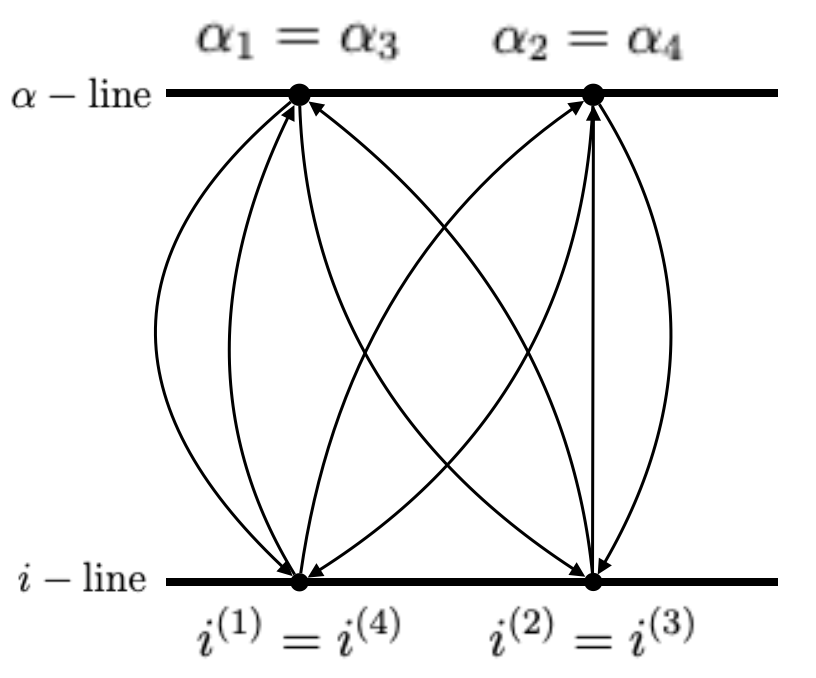}
\end{minipage}
}
\caption{}
\label{Fig-2}
\end{figure}

\begin{figure}[h]
    \centering
    \includegraphics[scale=0.35]{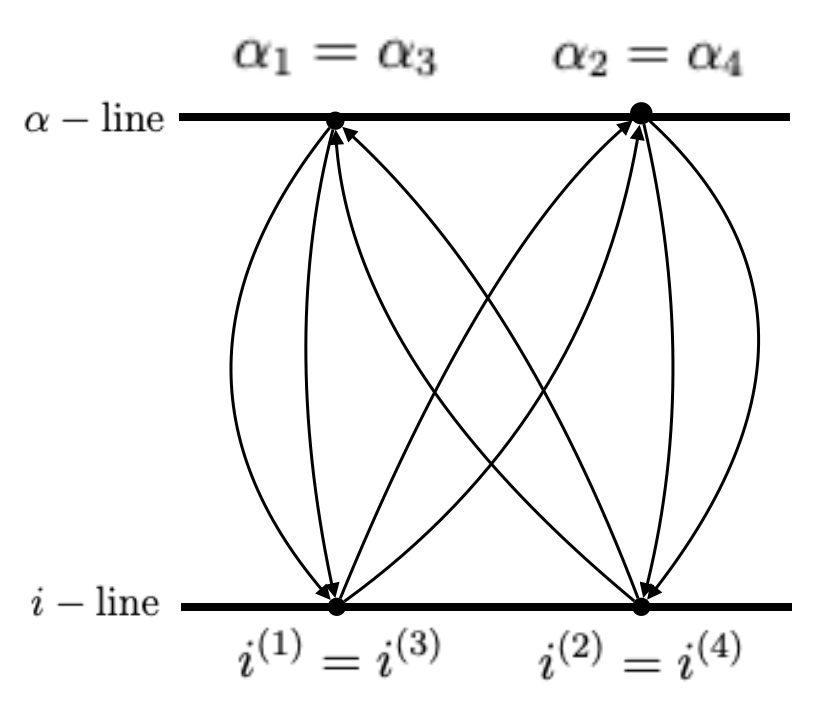}
    \caption{$g(i,\alpha)$ with $p=4$ and $\alpha=i=(1,2,1,2)$.}
    \label{Fig-3}
\end{figure}

Next, we establish the following proposition for the $\Delta(p;\alpha)$-graph for non-crossing sequence $\alpha$.

\begin{proposition} \label{Prop-non crossing}
For any $1 \le s \le p$, and $\alpha \in \cC_{s,p}^{(1)}$, for any canonical sequence $i$ of length $p$, the graph $g(i,\alpha)$ is either a paired graph or a single graph.
\end{proposition}

In order to prove Proposition \ref{Prop-non crossing}, we need to introduce the conception of \emph{consecutive} down (resp. up) edges.
Let $\alpha,i$ be two canonical sequences.
For any two coincide down edges $\alpha_{j_1} \to i^{(j_1)}$ and $\alpha_{j_2} \to i^{(j_2)}$ with some $1 \le j_1<j_2 \le p$, if all up edges $\{ i^{(j)} \to \alpha_{j+1}: j_1 \le j < j_2\}$ between the two down edges do not coincide with them (without considering the orientation), then we call the two down edges $\alpha_{j_1} \to i^{(j_1)}$ and $\alpha_{j_2} \to i^{(j_2)}$ are \emph{consecutive} down edges with \emph{distance} $j_2-j_1$.
Similarly, for any two coincide up edges $i^{(j_1)} \to \alpha_{j_1+1}$ and $i^{(j_2)} \to \alpha_{j_2+1}$ with $1 \le j_1<j_2 \le p$, if down edge $\alpha_j \to i^{(j)}$ does not coincide with them for any $j_1+1 \le j \le j_2$, then we call the two up edges \emph{consecutive} up edges with \emph{distance} $j_2-j_1$. In the graph given by Figure \ref{Fig-3}, the two coincident down edges $\alpha_1 \to i^{(1)}$ and $\alpha_3 \to i^{(3)}$ are consecutive down edges with distance 2, while the two coincident up edges $i^{(1)} \to \alpha_2$ and $i^{(3)} \to \alpha_4$ are consecutive up edges with distance 2.

\begin{proof} (of Proposition \ref{Prop-non crossing})
We prove by contradiction. We fix the sequence $\alpha \in \cC_{s,p}^{(1)}$. Assume that there exists a canonical sequence $i = (i^{(1)}, \ldots, i^{(p)})$, such that the graph $g(i,\alpha)$ is neither a paired graph nor a single graph. By definition, there exist two vertices, such that the numbers of the up and down edges between the two vertices are different by at lease two. Thus, there are consecutive up edges or consecutive down edges. We choose the pair of consecutive edges with the smallest distance and consider the case that they are up edges and down edges separately. If there are more than one pair of consecutive edges with the smallest distance, we can choose any one of them.

\noindent{\bf Case 1.} The pair of consecutive edges with smallest distance are down edges $\alpha_{j_1} \to i^{(j_1)}$ and $\alpha_{j_2} \to i^{(j_2)}$ with some $1 \le j_1<j_2 \le p$.

We restrict out attention to the path $P: \alpha_{j_1} \to i^{(j_1)} \to \alpha_{j_1+1} \to \ldots \to i^{(j_2-1)} \to \alpha_{j_2}$. For all vertices that coincides with $i^{(j_1)}$, we denote by $A$ the collection of their neighbourhoods among the collection $\{\alpha_j: j_1+1 \le j \le j_2-1\}$ and $E$ the corresponding collection of edges. We denote by $B$ the collection $\{\alpha_1, \ldots, \alpha_{j_1}, \alpha_{j_2}, \ldots, \alpha_{p+1}\}$. We keep the multiplicity for coincide vertices (resp. edges) for $A$ (resp. $E$).

Note that vertices in $A$ do not coincide with $\alpha_{j_1}$ by the definition of consecutive down edges, and do not coincide with any vertex in $B$ since $\alpha$ is non-crossing. Thus, $A$ and $B$ are disjoint.
Since the vertex $i^{(j_1)}$ is not the endpoint of the path $P$, the numbers of the up edges and down edges within $P$ associated with $i^{(j_1)}$ are the same. Noting that on the path $P$, the edge $\alpha_{j_1} \to i^{(j_1)}$ is the only edge associated with $i^{(j_1)}$ that are not in $E$, so the number of edges in $E$ is odd.
Hence, there exist the coincide edges in $E$ consist of different number of up edges and down edges. If the difference of up and down coincident edges is exactly one, then the graph is a single graph, which is a contradiction.
If the up and down coincident edges differ by at least two, then there is another pair of consecutive edges. This also leads to a contradiction since the consecutive down edges $\alpha_{j_1} \to i^{(j_1)}$ and $\alpha_{j_2} \to i^{(j_2)}$ should have the smallest distance.

\noindent{\bf Case 2.} The pair of consecutive edges with smallest distance are up edges $i^{(j_1)} \to \alpha_{j_1+1}$ and $i^{(j_2)} \to \alpha_{j_2+1}$ with $1 \le j_1<j_2 \le p$.

The argument is similar to the Case 1, and is sketched below. We consider the path $P': \alpha_{j_1+1} \to \ldots \to \alpha_{j_2} \to i^{(j_2)} \to \alpha_{j_2+1}$. For all vertices that coincides with $i^{(j_2)}$, we denote by $A'$ the collection of their neighbourhoods among $\{\alpha_j: j_1+2 \le j \le j_2\}$ and $E'$ the corresponding collection of edges. We also denote $B' = \{\alpha_1, \ldots, \alpha_{j_1+1}, \alpha_{j_2+1}, \ldots, \alpha_{p+1}\}$. Then by the definition of consecutive up edges and the fact that $\alpha$ is non-crossing, one can deduce that $A'$ and $B'$ are disjoint. Besides, by analyzing the neighbourhood of $i^{(j_2)}$ in the path $P'$, one can deduce that the number of edges in $E'$ is odd. This contradicts to either the condition that $g(i,\alpha)$ is not a single graph or the assumption that consecutive edges have distance at least $j_2-j_1$.
\end{proof}

\subsection{Paired graph} \label{sec:paired garph}

In this subsection, we study the paired graphs, which contribute to the moments in Section \ref{sec:moment}. For graphs that are not single graph, we have the following proposition for the number of vertices.

\begin{proposition} \label{Prop-number of vertices}
For any $1 \le r,s \le p$, for any $\alpha \in \cC_{s,p}$ and $i \in \cC_{r,p}$, we have the following statements:
\begin{enumerate}
    \item If $g(i,\alpha)$ is a paired graph, then $r+s \le p+1$. The equality holds if and only if $g(i,\alpha)$ is a $\Delta_1(p,s;\alpha)$-graph.
    \item If $g(i,\alpha)$ is neither a paired graph nor a single graph, then $r+s \le p$.
\end{enumerate}
\end{proposition}

\begin{proof}
If $g(i,\alpha)$ is not a single graph, then all edges must coincide with at least one other edges. If we remove the orientation and glue all the coincide edges, it results in non-directed connected graph with at most $p$ edges and exactly $r+s$ vertices, which implies that $r+s \le p+1$. The equality holds if and only if the resulting graph is a tree with exactly $p$ edges. In this case, all edges in the graph $g(i,\alpha)$ must coincide with exactly one other edge. If there are two coincident edges that have the same orientation, then directed graph $g(i,\alpha)$ is disconnected, which is a contradiction. Thus, the equality only happens when the graph $g(i,\alpha)$ is a $\Delta_1(p,s;\alpha)$-graph.
\end{proof}

Next, we introduce the \emph{Stirling number of the second kind} with the notation $S(n,k)$, which is defined as the number of ways to partition a set of $n$ objects into $k$ non-empty subsets. For non-crossing sequence $\alpha$, the following proposition counts the number of paired graph associate to $\alpha$.

\begin{proposition} \label{Prop-number of paired graph}
For any $1 \le s \le p$ and any sequence $\alpha \in \cC_{s,p}^{(1)}$, the number of sequence $i \in \cC_{r,p}$ such that $g(i,\alpha)$ is a paired graph is $S(p+1-s,r)$ if $r \le p+1-s$, and is 0 if $r > p+1-s$.
\end{proposition}

\begin{proof}
The case $r>p+1-s$ is straightforward from Proposition \ref{Prop-number of vertices}. In the following, we only consider the case $r \le p+1-s$. We fix a sequence $\alpha \in \cC_{s,p}^{(1)}$.

Firstly, we will show that any paired graph $g(i,\alpha)$ can be transferred to a $\Delta_1(p,s;\alpha)$-graph by splitting the vertices in the sequence $i$. By definition, one can easily see that paired graphs may have more than one pair of up and down edges between two vertices, and may have cycles if the multiple edges are glued and orientation are removed. Thus, our strategy is to remove the multiple pairs of up and down edges in the first step, and then remove the cycles in the second step.

\noindent{\bf Step 1.} For any two vertices $v_1,v_2$ in the paired graph $g(i,\alpha)$, we denote by $m_{v_1,v_2}$ the number of the up edges between $v_1$ and $v_2$. We define
\begin{align*}
    K(g(i,\alpha)) = \sum_{v_1,v_2} \left( m_{v_1,v_2}-1 \right)
\end{align*}
where the sum $\sum_{v_1,v_2}$ is over all pairs of vertices $(v_1,v_2)$ that are neighbourhood in $g(i,\alpha)$. One can easily check by definition that $K(g(i,\alpha)) = 0$ if and only if every edge in $g(i,\alpha)$ coincides with exactly one edge, and the two coincident edges have different orientation.

For the case $K(g(i,\alpha)) = 1$, there exists two vertices, between which there are two up edges and two down edges. In the following, we will split the corresponding $i$-vertex can into two vertices and resulting in a new $i$-sequence $i'$, such that the the graph $g(i',\alpha)$ is a paired graph without coincident edges of the same orientation. The argument is similar to \cite{BYY2022}*{Lemma 3.3}, and is sketched below in two cases.

\noindent {\bf Case 1.} If we scan the edges from $\alpha_1$ to $\alpha_{p+1}$, the first appearance of the four coincident edges is an down edges. In this case, the coincident edges are the $j$th down edge $\alpha_j \to i^{(j)}$, the $l$th down edge $\alpha_l \to i^{(l)}$, the $j'$th up edge $i^{(j')} \to \alpha_{j'+1}$ and the $l'$th up edge $i^{(l')} \to \alpha_{l'+1}$ for some $j<j'+1\le l<l'+1$. We split the vertex $i^{(l)}$ into two vertices $i^{(l,1)}$ and $i^{(l,2)}$. The edges from $\alpha_1 \to i^{(1)}$ to $i^{(l-1)} \to \alpha_l$ that connects $i^{(l)}$ are plotted to connect $i^{(l,1)}$, while the edges from $\alpha_l \to i^{(l)}$ to $i^{(p)} \to \alpha_{p+1}$ that connects $i^{(l)}$ are plotted to connect $i^{(l,2)}$. See Figure \ref{Fig-split-Case 1} below.

\begin{figure}[h]
    \centering
    \includegraphics[scale=0.5]{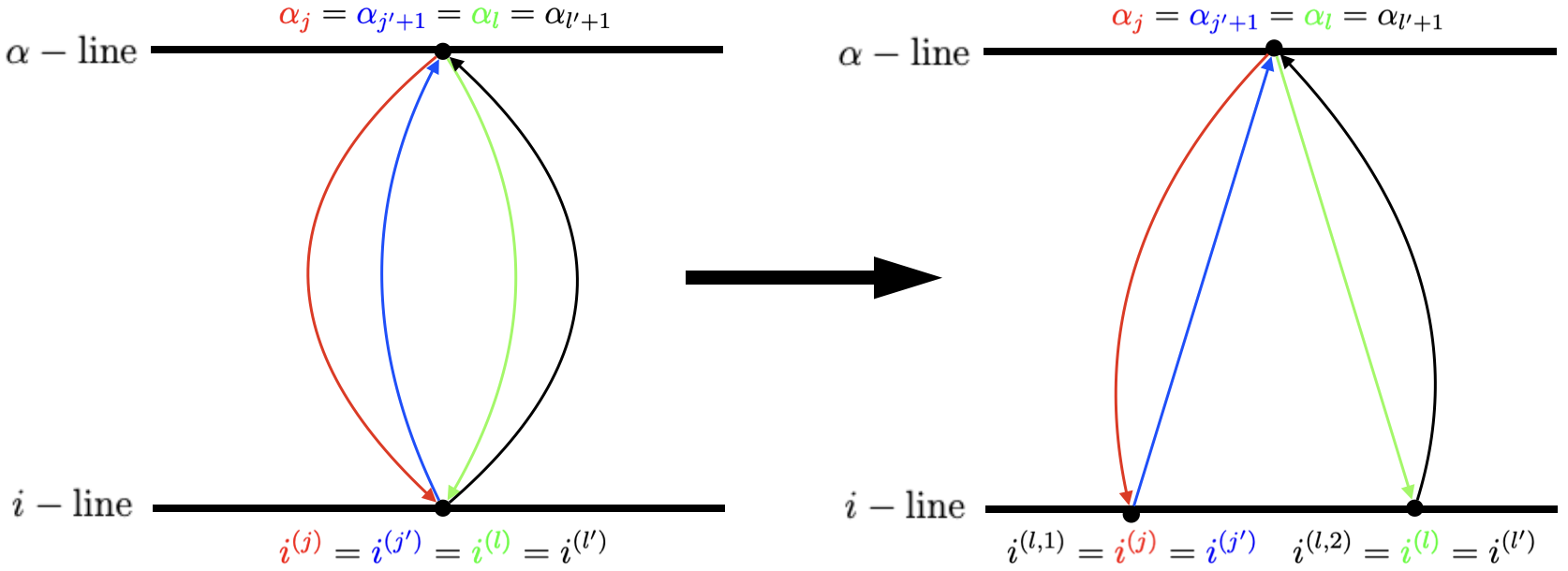}
    \caption{Case 1-Split an $i$ vertex to cancel multiple pairs of edges.}
    \label{Fig-split-Case 1}
\end{figure}

\noindent {\bf Case 2.} If we scan the edges starting from $\alpha_1 \to i^{(1)}$, the first appearance of the four coincident edges is an up edges. In this case, the coincident edges are the $j$th up edge $i^{(j)} \to \alpha_{j+1}$, the $l$th up edge $i^{(l)} \to \alpha_{l+1}$, the $j'$th down edge $\alpha_{j'} \to i^{(j')}$ and the $l'$th down edge $\alpha_{l'} \to i^{(l')}$ for some $j<j'\le l<l'$. We split the vertex $i^{(l)}$ into two vertices $i^{(l,1)}$ and $i^{(l,2)}$. The edges from $\alpha_{j'} \to i^{(j')}$ to $i^{(l)} \to \alpha_{l+1}$ that connects $i^{(l)}$ are plotted to connect $i^{(l,2)}$, while the rest of the edges that connects $i^{(l)}$ are plotted to connect $i^{(l,1)}$. See Figure \ref{Fig-split-Case 2} below.

\begin{figure}[h]
    \centering
    \includegraphics[scale=0.5]{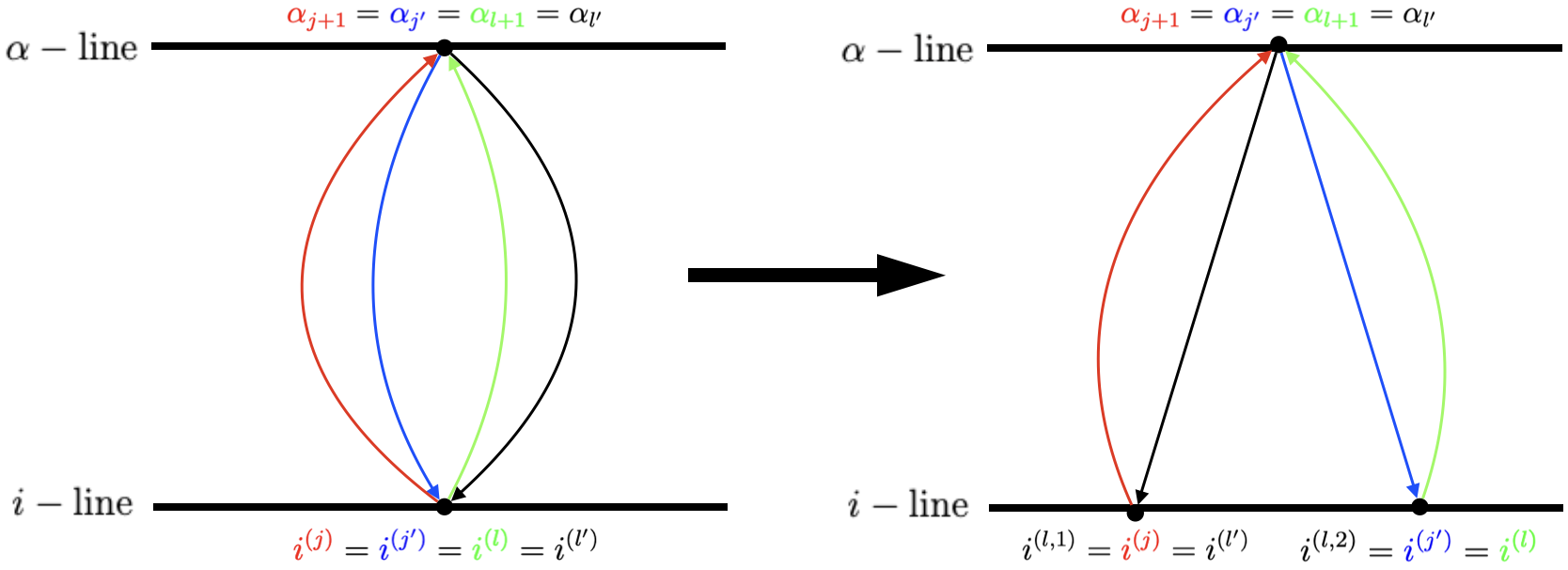}
    \caption{Case 2-Split an $i$ vertex to cancel multiple pairs of edges.}
    \label{Fig-split-Case 2}
\end{figure}

In both cases, one could check that after splitting the vertex $i^{(l)}$, the graph is still connected, and is paired graph. Moreover, the number of edges between any two vertices is either 0 or 2, which implies that $K(g(i',\alpha)) = 0$.

One can use induction to show that there exists a sequence $i'$, such that $K(g(i',\alpha)) = 0$ and the paired graph $g(i',\alpha)$ can be obtained from $g(i,\alpha)$ by splitting some of the vertices in $i$. Indeed, by scanning all edges starting from $\alpha_1 \to i^{(1)}$, we can find the first coincident directed edges. Then we can apply the argument above to split the $i$-vertex associated to the coincident directed edges into two $i$-vertices. We denote the resulting $i$-sequence by $\tilde i'$. Then we have $K(g(\tilde i',\alpha)) = K(g(i,\alpha)) - 1$. By the induction hypothesis, we can find a sequence $i'$ by splitting $\tilde i'$, such that $K(g(i',\alpha)) = 0$. Moreover, the $i$-sequence $i'$ can also be obtained by splitting the sequence $i$.

\noindent{\bf Step 2.} Let $g(i',\alpha)$ be a paired graph such that $K(g(i',\alpha))= 0$. Then every up edge coincides with exactly one down edge. Denote by $C(g(i',\alpha))$ the number of cycles when gluing all the pairs of coincident up edge and down edge and removing the orientation. By definition, $C(g(i',\alpha)) = 0$ if and only if $g(i',\alpha)$ is a $\Delta_1(p,s;\alpha)$-graph.

If $C(g(i',\alpha)) = 1$, then there is exactly one cycle when gluing the pair of coincident up edge and down edge. In the following, we will split one vertex in $i'$ and denote by $i''$ the new $i$-sequence, such that $g(i'',\alpha)$ is still a paired graph without any cycle when gluing all pairs of coincident up edge and down edge, and $g(i'',\alpha)$ does not have coincident edges of the same orientation. That is, $g(i'',\alpha)$ is a $\Delta_1(p,s;\alpha)$-graph. The argument is similar to \cite{BYY2022}*{Lemma 3.6}, and is sketched below.

We scan the edges starting from $\alpha_1 \to i^{(1)}$, and find the first edge that results in a cycle without considering orientation and removing the multiplicity of the edges. Using the non-crossing property of $\alpha$, one can show that this edge must be a down edge. We denote the down edge by $\alpha_j \to i^{(j)}$. We split the vertex $i^{(j)}$ into two vertices $i^{(j,1)}$ and $i^{(j,2)}$. The edges in the path $\alpha_1 \to i^{(1)} \to \ldots \to \alpha_j$ that connects $i^{(j)}$ are plotted to connect $i^{(j,1)}$, while the edges in the path $\alpha_j \to i^{(j)} \to \ldots \to i^{(p)} \to \alpha_{p+1}$ that connects $i^{(j)}$ are plotted to connect $i^{(j,2)}$. See Figure \ref{Fig-split-cycle}. Once could check that after splitting the vertex $i^{(j)}$, there is no multiple up edge or multiple down edge. Besides, there is no cycle without considering the orientation and gluing pairs of coincident up edge and down edge. Hence, if we denote by $i''$ the new $i$-sequence, then $K(g(i',\alpha)) = 0 = C(g(i',\alpha))$.

\begin{figure}[h]
    \centering
    \includegraphics[scale=0.5]{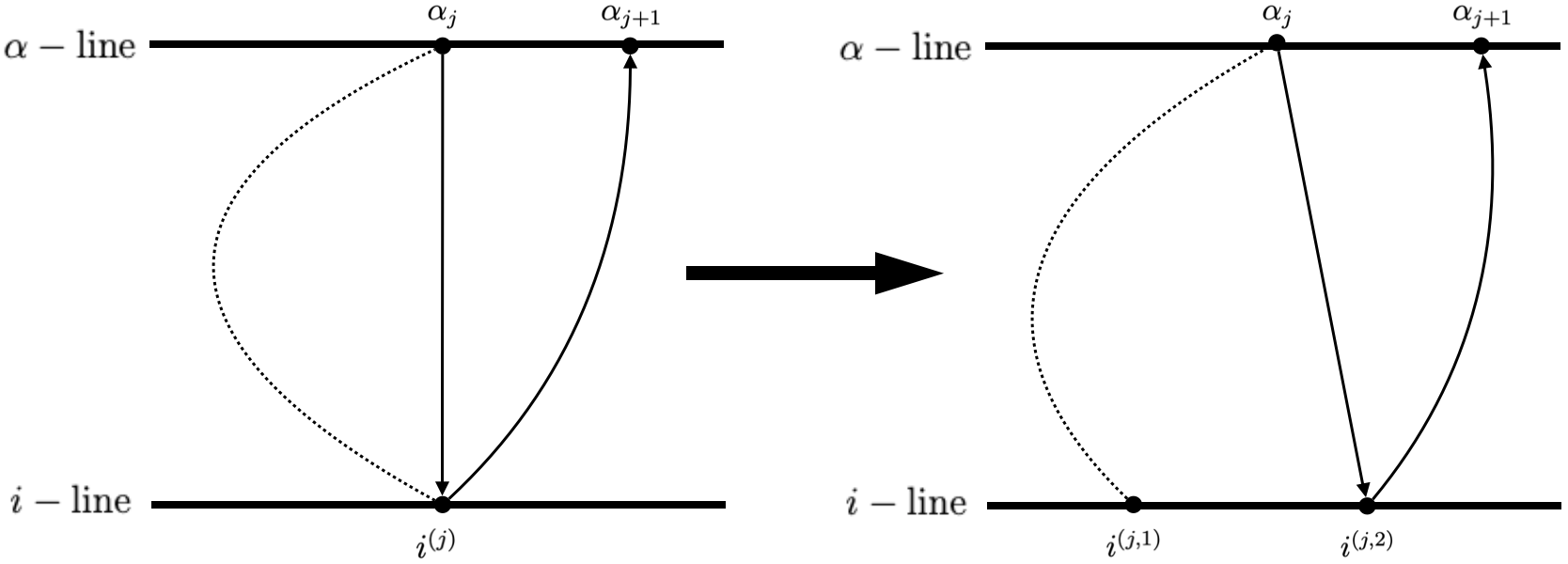}
    \caption{Split an $i$ vertex to cancel cycle.}
    \label{Fig-split-cycle}
\end{figure}

One can use induction to show that there exists a sequence $i''$, such that $C(g(i'',\alpha)) = K(g(i'',\alpha)) = 0$ and the paired graph $g(i'',\alpha)$ can be obtained from $g(i',\alpha)$ by splitting some of the vertices in $i'$. Indeed, we can scan all edges from $\alpha_1 \to i'^{(1)}$ and find the first edge which forms a cycle when gluing coincident edges and removing orientation. Then we can use the argument above to split one of the $i$-vertex in the cycle into two $i$-vertices. We denote the resulting $i$-sequence by $\tilde i''$. The splitting procedure will not lead to coincident up edges nor coincident down edges, nor new cycle when gluing all pairs of coincident edges. Thus, we have $C(g(\tilde i'',\alpha)) \le C(g(i',\alpha)) - 1$ and $K(g(\tilde i'',\alpha)) = 0$. Then by induction hypothesis, we can split vertices in $\tilde i''$ to obtain $i''$, such that $g(i'',\alpha)$ is paired graph and $C(g(i'',\alpha)) = K(g(i'',\alpha)) = 0$. Moreover, the $i$-sequence $i''$ can also be obtained by splitting vertices in the sequence $i'$.

Therefore, joining the two steps above, for any paired graph $g(i,\alpha)$, we can split vertices on $i$ to obtain $i''$, such that $g(i'',\alpha)$ is a $\Delta_1(p,s;\alpha)$-graph.

Secondly, we will establish a bijective map from the set of all partitions of $[p+1-s]$ to the set of the canonical $r$-sequence that form a paired graph with $\alpha$.

By Lemma \ref{lem-Bai}, there exists a unique canonical $(p+1-s)$-sequence $i = (i^{(1)}, \ldots, i^{(p)})$, such that $g(i,\alpha)$ is a $\Delta_1(p,s;\alpha)$-graph.
Let $\cP(p+1-s)$ be the set of all partitions of $[p+1-s]$, and $\cP(p+1-s,q)$ be the set of all partitions of $[p+1-s]$ with $q$ blocks. For a partition $\pi \in \cP(p+1-s,q)$ with blocks $V_1,V_2,\ldots, V_q$, without loss of generality, we assume that
\[ \min\{a: a \in V_1\} < \ldots < \min\{a: a \in V_q\}. \]
We identify the partition $\pi$ with the mapping $\pi: [p+1-s] \to [q]$ given by $\pi(a) = b$ if $a \in V_{b}$. We abuse the notation for partition and the corresponding mapping. By the definition, it is easy to see that $\pi$ maps a canonical sequence to a canonical sequence.
For fixed $\alpha \in \cC_{s,p}^{(1)}$, let $\cP'(\alpha)$ be the set of all canonical sequences $i'$ such that $g(i',\alpha)$ are paired graphs. We write $\cP'(\alpha,r)$ for the set of all canonical $r$-sequences $i'$ in $\cP'(\alpha)$. We consider the following mapping:
\begin{align*}
    \Phi: \cP(p+1-s) \quad &\longrightarrow \quad \cP'(\alpha) \\
    \pi \quad\quad\quad &\longrightarrow \quad \pi(i),
\end{align*}
where $\pi(i) = (\pi(i^{(1)}), \ldots, \pi(i^{(p)}))$. 

Note that for any $\pi \in \cP(p+1-s)$, $g(\pi(i),\alpha)$ can be obtained from $g(i,\alpha)$ by gluing the $i$-vertices according to the partition $\pi$. Since $g(i,\alpha)$ is a paired graph,  so is $g(\pi(i),\alpha)$, which implies that $\Phi$ is well-defined. As we have proved in the first part that any paired graph can be transferred to a $\Delta_1(p,s;\alpha)$-graph by appropriately splitting the vertices in the $i$-sequence, we can conclude that $\Phi$ is surjective. Moreover, for two different partitions $\pi_1, \pi_2 \in \cP(p+1-s)$, the two canonical sequence $\pi_1(i)$ and $\pi_2(i)$ are different. Thus, $\Phi$ is injective. Therefore, $\Phi$ is bijective.

To conclude, we consider the following restriction of $\Phi$:
\begin{align*}
    \Phi|_r: \cP(p+1-s,r) \quad &\longrightarrow \quad \cP'(\alpha,r).
\end{align*}
Since the bijectivity of $\Phi|_r$ inherites from $\Phi$, by the definition of Stirling number of the second kind, we have
\begin{align*}
    \# \cP'(\alpha,r) = \# \cP(p+1-s,r) = S(p+1-s,r).
\end{align*}
\end{proof}

We end this subsection by collecting some properties of the Stirling number of the second kind. We refer the readers to \cite{Grimaldi} for more details.

\begin{lemma} \label{Lem-Stirling number}
\begin{enumerate}
    \item We have $S(n,n) = S(n,1) =1$ for $n \ge 1$. For $1 \le k \le n$, we have
    \begin{align*}
        S(n,k) = \sum_{i=1}^k \dfrac{(-1)^{k-i} i^n}{i!(k-i)!}.
    \end{align*}
    \item For positive integers $n \ge k >1$, we have
    \begin{align*}
        S(n+1,k) = S(n,k-1) + k S(n,k).
    \end{align*}
    \item For positive integers $n$, we have
    \begin{align*}
        \sum_{k=1}^n S(n,k) \cdot x(x-1) \ldots (x-k+1) = x^n.
    \end{align*}
\end{enumerate}
\end{lemma}

\section{Convergence of spectral moments}
\label{sec:moment}

\subsection{Proof of Theorem \ref{Thm-main}} \label{sec:(a)}
We compute the moment
\begin{align*}
	\dfrac{1}{n^k} \bE \big[ \Tr M_{n,k,m}^p \big].
\end{align*}
for any $p \in \bN_+$.
By convention,  $\alpha_{p+1} = \alpha_{1}$. We have
\begin{align} \label{eq-moment-0.2}
	\dfrac{1}{n^k} \bE \big[ \Tr M_{n,k,m}^p \big]
	=& \dfrac{1}{n^k} \sum_{\alpha_1, \ldots, \alpha_p = 1}^m \left( \prod_{t=1}^p \tau_{\alpha_t} \right) \bE \Big[ \Tr \big( Y_{\alpha_1}^* Y_{\alpha_2} Y_{\alpha_2}^* \cdots Y_{\alpha_p} Y_{\alpha_p}^* Y_{\alpha_{p+1}} \big) \Big] \nonumber \\
	=& \dfrac{1}{n^k} \sum_{\alpha_1, \ldots, \alpha_p = 1}^m \left( \prod_{t=1}^p \tau_{\alpha_t} \right) \bE \Bigg[ \prod_{l=1}^k \Tr \left( \left( \y_{\alpha_{1}}^{(l)} \right)^* \y_{\alpha_{2}}^{(l)} \left( \y_{\alpha_{1}}^{(l)} \right)^* \ldots \y_{\alpha_{p+1}}^{(l)} \right) \Bigg] \nonumber \\
	=& \dfrac{1}{n^k} \sum_{\alpha_1, \ldots, \alpha_p = 1}^m \left( \prod_{t=1}^p \tau_{\alpha_t} \right) \left( \bE \Bigg[ \Tr \left( \left( \y_{\alpha_{1}}^{(1)} \right)^* \y_{\alpha_{2}}^{(1)} \left( \y_{\alpha_{1}}^{(1)} \right)^* \ldots \y_{\alpha_{p+1}}^{(1)} \right) \Bigg] \right)^k \nonumber \\
	=& \dfrac{1}{n^k} \sum_{\alpha_1, \ldots, \alpha_p = 1}^m \left( \prod_{t=1}^p \tau_{\alpha_t} \right) \left( \bE \left[ \sum_{i^{(1)}, \ldots, i^{(p)} = 1}^n \prod_{t=1}^p \bigg( \overline{\big( \y_{\alpha_t}^{(1)} \big)_{i^{(t)}}} \big( \y_{\alpha_{t+1}}^{(1)} \big)_{i^{(t)}} \bigg) \right] \right)^k,
\end{align}
where we used the i.i.d. setting in the third equality.

For two sequences $\alpha = (\alpha_1, \ldots, \alpha_p) \in [m]^p$ and $i = (i^{(1)}, \ldots, i^{(p)}) \in [n]^p$, let 
\begin{align} \label{eq-def-E(i_1,alpha)}
	E(i,\alpha) = \bE \left[ \prod_{t=1}^p \bigg( \overline{\big( \y_{\alpha_t}^{(1)} \big)_{i^{(t)}}} \big( \y_{\alpha_{t+1}}^{(1)} \big)_{i^{(t)}} \bigg) \right].
\end{align}
By the i.i.d. setting, $E(i,\alpha)=E(i',\alpha')$ if the two sequences $i$ and $\alpha$ are equivalent to $i'$ and $\alpha'$,  respectively. By \eqref{eq-moment-0.2} and \eqref{eq-0.3}, we have
\begin{align} \label{eq-moment-0.4}
	\dfrac{1}{n^k} \bE \big[ \Tr M_{n,k,m}^p \big]
	=& \dfrac{1}{n^k} \sum_{s=1}^p \sum_{\alpha \in \cJ_{s,p}(m)} \left( \prod_{t=1}^p \tau_{\alpha_t} \right) \left( \sum_{r=1}^p \sum_{i \in \cJ_{r,p}(n)} E(i,\alpha) \right)^k \nonumber \\
	=& \dfrac{1}{n^k} \sum_{s=1}^p \sum_{\alpha \in \cC_{s,p}} \left( \sum_{\varphi \in \cI_{s,m}} \prod_{t=1}^p \tau_{\varphi(\alpha_t)} \right) \left( \sum_{r=1}^p n \cdots (n-r+1) \sum_{i \in \cC_{r,p}} E(i,\alpha) \right)^k \nonumber \\
	:=& I_1 + I_2,
\end{align}
where
\begin{align*}
    I_1 =& \dfrac{1}{n^k} \sum_{s=1}^p \sum_{\alpha \in \cC_{s,p}^{(1)}} \left( \sum_{\varphi \in \cI_{s,m}} \prod_{t=1}^p \tau_{\varphi(\alpha_t)} \right) \left( \sum_{r=1}^p n \cdots (n-r+1) \sum_{i \in \cC_{r,p}} E(i,\alpha) \right)^k, \\
    I_2 =& \dfrac{1}{n^k} \sum_{s=1}^p \sum_{\alpha \in \cC_{s,p} \setminus \cC_{s,p}^{(1)}} \left( \sum_{\varphi \in \cI_{s,m}} \prod_{t=1}^p \tau_{\varphi(\alpha_t)} \right) \left( \sum_{r=1}^p n \cdots (n-r+1) \sum_{i \in \cC_{r,p}} E(i,\alpha) \right)^k.
\end{align*}

Note that the component of the base vector satisfies
\begin{align} \label{eq-base}
    \left( \y_{\beta}^{(l)} \right)_i \overline{\left( \y_{\beta}^{(l)} \right)_i} = \dfrac{1}{n}, \quad \left| \bE \left[ \left( \left( \y_{\beta}^{(l)} \right)_i \right)^p \right] \right| \le \dfrac{1}{n^{p/2}}.
\end{align}
Recall the definition of paired graph and single graph in Definition \ref{Def-paired}. For any sequence $\alpha \in \cC_{s,p}$ and $i\in \cC_{r,p}$, we have
\begin{align} \label{eq-E(i,alpha)}
    \begin{cases}
    E(i,\alpha) = n^{-p}, & g(i,\alpha) \mathrm{\ is \ a \ paired \ graph,} \\
    E(i,\alpha) =0, & g(i,\alpha) \mathrm{\ is \ a \ single \ graph,} \\
    |E(i,\alpha)| \le n^{-p}, & g(i,\alpha) \mathrm{\ otherwise.} \\
    \end{cases}
\end{align}

Firstly, we deal with $I_1$. For any $\alpha \in \cC_{s,p}^{(1)}$, by Proposition \ref{Prop-non crossing}, formula \eqref{eq-E(i,alpha)} and Proposition \ref{Prop-number of paired graph}, we obtain
\begin{align*}
    \sum_{i \in \cC_{r,p}} E(i,\alpha)
    =& n^{-p} \cdot \#\{i \in \cC_{r,p}: g(i,\alpha) \ \mathrm{is \ paired \ graph} \} \\
    =& \begin{cases}
        n^{-p} \cdot S(p+1-s,r), & r \le p+1-s, \\
        0, & r>p+1-s,
    \end{cases}
\end{align*}
where we use the notation $\#S$ for the number of elements in the set $S$. Thus, it follows from Lemma \ref{Lem-Stirling number} that
\begin{align*}
    \sum_{r=1}^p n \cdots (n-r+1) \sum_{i \in \cC_{r,p}} E(i,\alpha)
    = n^{-p} \sum_{r=1}^{p+1-s} n \cdots (n-r+1) \cdot S(p+1-s,r)
    = n^{1-s}.
\end{align*}
Hence,
\begin{align} \label{eq-I_1}
    I_1 = \sum_{s=1}^p \left( \dfrac{m}{n^k} \right)^s \sum_{\alpha \in \cC_{s,p}^{(1)}} \left( \dfrac{1}{m^s} \sum_{\varphi \in \cI_{s,m}} \prod_{t=1}^p \tau_{\varphi(\alpha_t)} \right).
\end{align}

Next, we deal with $I_2$. For any $\alpha \in \cC_{s,p} \setminus \cC_{s,p}^{(1)}$, by Lemma \ref{lem-Bai} and Proposition \ref{Prop-number of vertices}, if the graph $g(i,\alpha)$ is not a single graph for $i \in \cC_{r,p}$, then $r+s \le p$. Hence, by \eqref{eq-E(i,alpha)}, we establish
\begin{align*}
    \left| \sum_{r=1}^p n \cdots (n-r+1) \sum_{i \in \cC_{r,p}} E(i,\alpha) \right|
    \le& \sum_{r=1}^{p-s} n^{-p+r} \cdot \#\{i \in \cC_{r,p}: g(i,\alpha) \ \mathrm{is \ not \ a \ single \ graph} \}.
\end{align*}
Thus, we have
\begin{align} \label{eq-I_2}
    |I_2| \le& \dfrac{1}{n^k} \sum_{s=1}^p \sum_{\alpha \in \cC_{s,p} \setminus \cC_{s,p}^{(1)}} \left| \sum_{\varphi \in \cI_{s,m}} \prod_{t=1}^p \tau_{\varphi(\alpha_t)} \right| \left| \sum_{r=1}^p n \cdots (n-r+1) \sum_{i \in \cC_{r,p}} E(i,\alpha) \right|^k \nonumber \\
    =& \sum_{s=1}^p \left( \dfrac{m}{n^k} \right)^s \sum_{\alpha \in \cC_{s,p} \setminus \cC_{s,p}^{(1)}} \left| \dfrac{1}{m^s} \sum_{\varphi \in \cI_{s,m}} \prod_{t=1}^p \tau_{\varphi(\alpha_t)} \right| \nonumber \\
    & \times \left| \sum_{r=1}^{p-s} n^{-p+r+s-1} \cdot \#\{i \in \cC_{r,p}: g(i,\alpha) \ \mathrm{is \ not \ a \ single \ graph} \} \right|^k.
\end{align}
Note that the assumption \eqref{eq-condition-moment convergence} implies the following convergence:
\begin{align*}
    \dfrac{1}{m^s} \sum_{\varphi \in \cI_{s,m}} \prod_{t=1}^p \tau_{\varphi(\alpha_t)}
    \to \prod_{t=1}^s m_{\deg_t(\alpha)}^{(\tau)}, \quad m \to \infty.
\end{align*}
Hence, under the limiting setting \eqref{eq-def-ratio}, when $n,k \to \infty$, we can deduce from \eqref{eq-I_1} and \eqref{eq-I_2} that
\begin{align} \label{eq-I limit}
    I_1 \to \sum_{s=1}^p c^s \sum_{\alpha \in \cC_{s,p}^{(1)}} \left( \prod_{t=1}^s m_{\deg_t(\alpha)}^{(\tau)} \right),
    \quad I_2 \to 0.
\end{align}
Therefore, the proof is concluded by taking limit $n,k \to \infty$ in \eqref{eq-moment-0.4} and using \eqref{eq-I limit}.

\subsection{Proof of Theorem \ref{Thm-main2}} \label{sec:(b)}

For any $p \in \bN$, for $k\ge 2$, we compute the variance
\begin{align*}
    \Var \left( \dfrac{1}{n^k} \Tr M_{n,k,m}^p \right).
\end{align*}
The idea is similar to \cites{BYY2022}, and is sketched below. By the computation of \cite{BYY2022}*{Section 3.2}, we have
\begin{align*}
    \Var \left( \dfrac{1}{n^k} \Tr M_{n,k,m}^p \right)
    =& \dfrac{1}{n^{2k}} \sum_{\alpha,\beta \in [m]^p \atop \alpha \cap \beta \not= \emptyset} \left( \prod_{t=1}^p \tau_{\alpha_t} \tau_{\beta_t} \right) \\
    & \times \left[ \left( \sum_{i,j \in [n]^p} E'(i,\alpha;j,\beta) \right)^k - \left( \sum_{i,j \in [n]^p} E(i,\alpha) E(j,\beta) \right)^k \right],
\end{align*}
where $E(\cdot,\cdot)$ is given in \eqref{eq-E(i,alpha)}, and $E'(i,\alpha;j,\beta)$ is defined by
\begin{align*}
    E'(i,\alpha;j,\beta) = \bE \left[ \prod_{t=1}^p \bigg( \overline{\big( \y_{\alpha_t}^{(1)} \big)_{i^{(t)}}} \big( \y_{\alpha_{t+1}}^{(1)} \big)_{i^{(t)}} \overline{\big( \y_{\beta_t}^{(1)} \big)_{j^{(t)}}} \big( \y_{\beta_{t+1}}^{(1)} \big)_{j^{(t)}} \bigg) \right].
\end{align*}

Next, we join the two graphs $g(i,\alpha)$ and $g(j,\beta)$ together and keep the coincident edges. We denote by $g(i,\alpha) \cup g(j,\beta)$ the resulting graph. If there is an edge in the graph $g(i,\alpha) \cup g(j,\beta)$ that does not coincide with any other edges, then this edge must belong to $g(i,\alpha)$ or $g(j,\beta)$, which implies
\begin{align*}
    E'(i,\alpha;j,\beta) = E(i,\alpha) E(j,\beta) = 0.
\end{align*}
Thus, we only need to consider the indices such that all edges in $g(i,\alpha) \cup g(j,\beta)$ coincide with other edges. Noting that $\alpha \cap \beta \not = \emptyset$, the graph $g(i,\alpha) \cup g(j,\beta)$ is connected with $4p$ edges. Hence, if we remove orientation and glue coincident edges for the graph $g(i,\alpha) \cup g(j,\beta)$, it results in a non-directed connected graph with at most $2p$ edges, which implies that
\begin{align*}
    |(\alpha,\beta)| + |(i,j)| \le 2p+1.
\end{align*}
Hence, we have
\begin{align*}
    \Var \left( \dfrac{1}{n^k} \Tr M_{n,k,m}^p \right)
    =& \dfrac{1}{n^{2k}} \sum_{s=1}^{2p} \sum_{(\alpha,\beta) \in \cC_{s,2p} \atop \alpha \cap \beta \not= \emptyset} \left( \sum_{\varphi \in \cI_{s,m}} \prod_{t=1}^p \tau_{\varphi(\alpha_t)} \tau_{\varphi(\beta_t)} \right) \\
    & \times \left[ \left( \sum_{r=1}^{2p+1-s} n\ldots (n-r+1) \sum_{(i,j) \in \cC_{r,2p}} E'(i,\alpha;j,\beta) \right)^k \right. \\
    & \left. \quad\quad\quad - \left( \sum_{r=1}^{2p+1-s} n\ldots (n-r+1) \sum_{(i,j) \in \cC_{r,2p}} E(i,\alpha) E(j,\beta) \right)^k \right].
\end{align*}
By \eqref{eq-base}, for any sequence $\alpha,\beta,i,j$, it holds that
\begin{align*}
    |E'(i,\alpha;j,\beta)|, \ |E(i,\alpha) E(j,\beta)| \le n^{-2p}.
\end{align*}
Thus,
\begin{align*}
    & \max \left\{ \left| \sum_{r=1}^{2p+1-s} n\ldots (n-r+1) \sum_{(i,j) \in \cC_{r,2p}} E(i,\alpha) E(j,\beta) \right|, \right. \\
    & \left. \quad\quad\quad\quad\quad\quad\quad\quad \left| \sum_{r=1}^{2p+1-s} n\ldots (n-r+1) \sum_{(i,j) \in \cC_{r,2p}} E'(i,\alpha;j,\beta) \right| \right\} \\
    \le& \sum_{r=1}^{2p+1-s} n^{r-2p} \sum_{(i,j) \in \cC_{r,2p}} 1
    \le C_p n^{1-s} (1+o_n(1)),
\end{align*}
where $C_p$ is a positive number that only depends on $p$, and $o_n(1)$ is a quantity that tends to 0 as $n \to \infty$.
Hence, we obtain
\begin{align*}
    \Var \left( \dfrac{1}{n^k} \Tr M_{n,k,m}^p \right)
    \le& \dfrac{2C_p^k}{n^{k}} \sum_{s=1}^{2p} \sum_{(\alpha,\beta) \in \cC_{s,2p} \atop \alpha \cap \beta \not= \emptyset} \left( \sum_{\varphi \in \cI_{s,m}} \prod_{t=1}^p \tau_{\varphi(\alpha_t)} \tau_{\varphi(\beta_t)} \right) n^{-ks} (1+o_n(1))^k \\
    =& \dfrac{2C_p^k}{n^{k}} \sum_{s=1}^{2p} \left( \dfrac{m}{n^k} \right)^s \sum_{(\alpha,\beta) \in \cC_{s,2p} \atop \alpha \cap \beta \not= \emptyset} \left( \sum_{\varphi \in \cI_{s,m}} \prod_{t=1}^s \dfrac{1}{m} \tau_{\varphi(t)}^{\deg_t(\alpha)+\deg_t(\beta)} \right) (1+o_n(1))^k.
\end{align*}
By assumption \eqref{eq-condition-moment convergence}, we have
\begin{align*}
    \sum_{\varphi \in \cI_{s,m}} \prod_{t=1}^s \dfrac{1}{m} \tau_{\varphi(t)}^{\deg_t(\alpha)+\deg_t(\beta)}
    \to \prod_{t=1}^s m_{\deg_t(\alpha)+\deg_t(\beta)}^{(\tau)}, \quad m \to \infty.
\end{align*}
Together with \eqref{eq-def-ratio}, we establish
\begin{align*}
    \Var \left( \dfrac{1}{n^k} \Tr M_{n,k,m}^p \right)
    \le& \dfrac{2C_p^k}{n^{k}} \sum_{s=1}^{2p} c^s \sum_{(\alpha,\beta) \in \cC_{s,2p} \atop \alpha \cap \beta \not= \emptyset} \left( \prod_{t=1}^s m_{\deg_t(\alpha)+\deg_t(\beta)}^{(\tau)} \right) (1+o_n(1))^{k+1+s}.
\end{align*}
Therefore, for $k \ge 2$, we have
\begin{align*}
    \sum_{n \ge 2} \Var \left( \dfrac{1}{n^k} \Tr M_{n,k,m}^p \right) < +\infty.
\end{align*}
The proof is concluded by Borel-Cantelli's Lemma, noting that $k = k(n)$ tends to infinity as $n \to \infty$.

\subsection{Proof of Corollary \ref{Coro}} \label{sec:(c)}

We start with the uniqueness of $\mu$. We have
\begin{align*}
    \left| \sum_{s=1}^p c^s \sum_{\alpha \in \cC_{s,p}^{(1)}} \left( \prod_{t=1}^s m_{\deg_t(\alpha)}^{(\tau)} \right) \right|
    \le& \sum_{s=1}^p c^s \sum_{\alpha \in \cC_{s,p}^{(1)}} \left( \prod_{t=1}^s A^{\deg_t(\alpha)} {\deg_t(\alpha)}^{\deg_t(\alpha)} \right) \\
    \le& \sum_{s=1}^p c^s \sum_{\alpha \in \cC_{s,p}^{(1)}} A^{\sum_{t=1}^s \deg_t(\alpha)} \left( \sum_{s=1}^p \deg_t(\alpha) \right)^{\sum_{s=1}^p \deg_t(\alpha)}.
\end{align*}
By definition \eqref{eq:deg_t}, for $\alpha \in \cC_{s,p}$, we have $\sum_{t=1}^s \deg_t(\alpha) = p$. Hence, it follows from Lemma \ref{lem-Bai} that
\begin{align*}
    \left| \sum_{s=1}^p c^s \sum_{\alpha \in \cC_{s,p}^{(1)}} \left( \prod_{t=1}^s m_{\deg_t(\alpha)}^{(\tau)} \right) \right|
    \le A^p p^p \sum_{s=1}^p c^s \sum_{\alpha \in \cC_{s,p}^{(1)}} 1
    = A^p p^p \sum_{s=1}^p \dfrac{c^s}{p} \binom{p}{s-1} \binom{p}{s}.
\end{align*}
Using the inequality $\binom{p}{s} \le 2^p$ for all $0 \le s \le p$, we obtain
\begin{align*}
    \left| \sum_{s=1}^p c^s \sum_{\alpha \in \cC_{s,p}^{(1)}} \left( \prod_{t=1}^s m_{\deg_t(\alpha)}^{(\tau)} \right) \right|
    \le 4^p A^p p^p \sum_{s=1}^p \dfrac{c^s}{p}
    \le 4^p A^p (1+c)^p p^p.
\end{align*}
Hence,
\begin{align*}
    \sum_{p=1}^{\infty} \left| \sum_{s=1}^p c^s \sum_{\alpha \in \cC_{s,p}^{(1)}} \left( \prod_{t=1}^s m_{\deg_t(\alpha)}^{(\tau)} \right) \right|^{-1/p}
    \ge \sum_{p=1}^{\infty} \left( 4^p A^p (c+1)^p p^p \right)^{-1/p} = \sum_{p=1}^{\infty} \dfrac{1}{4A(c+1)p} = +\infty.
\end{align*}
Therefore, the Carleman's condition is satisfied, which implies that there exists a unique probability measure $\mu$ whose moments are given by \eqref{eq-moment}.

The uniqueness of the probability measure $\mu$ corresponding to the moments in \eqref{eq-moment} guarantees the almost sure convergence of the ESD of $M_{n,k,m}$ towards $\mu$.

If $\tau_{\alpha} = 1$ for all $1 \le \alpha \le m$, then the condition \eqref{eq-condition-moment convergence} holds with $m_q^{(\tau)} = 1$ for all $q \in \bN$. In this case, the moment sequence \eqref{eq-moment} for $\mu$ becomes
\begin{align*}
    \int_{\bR} x^p \mu(dx)
   = \sum_{s=1}^p c^s \sum_{\alpha \in \cC_{s,p}^{(1)}} 1
   = \sum_{s=1}^p \dfrac{c^s}{p} \binom{p}{s-1} \binom{p}{s}, \quad \forall p \in \bN_+,
\end{align*}
where we use Lemma \ref{lem-Bai} in the last equality. By \cite{Bai2010}*{Lemma 3.1}, this moment sequence coincides with the moment sequence of the Mar\v{c}enko-Pastur law \eqref{eq-MP law}. Therefore, the uniqueness of $\mu$ implies that $\mu$ is exactly the Mar\v{c}enko-Pastur law \eqref{eq-MP law}.

\section*{Acknowledgments} 

The author gratefully acknowledges the financial support of ERC Consolidator Grant 815703 ”STAM- FORD: Statistical Methods for High Dimensional Diffusions”. Besides, the author would like to acknowledge Tiefeng Jiang, Felix Parraud, Kevin Schnelli, Jianfeng Yao for the discussion and helpful suggestions.

\bibliographystyle{plain}
\bibliography{tensor}

\end{document}